\documentclass{amsart}
\usepackage{amscd}
\usepackage{amssymb}
\usepackage[pagebackref]{hyperref}

\newcommand{\Z}{{\mathbb Z}}
\newcommand{\Q}{{\mathbb Q}}

\newcommand{\fa}{\mathfrak a}

\newcommand{\fm}{\mathfrak m}
\newcommand{\fl}{\mathfrak l}
\newcommand{\fp}{\mathfrak p}

\newcommand{\fr}{\mathfrak r}

\newcommand{\fq}{\mathfrak q}

\newcommand{\cO}{{\mathcal O}}

\newcommand{\Cl}{\operatorname{Cl}}
\newcommand{\eps}{\varepsilon}

\newcommand{\Llra}{\Longleftrightarrow}

\newcounter{Tc}
\newcounter{Pc}
\newcounter{Lc}

\newtheorem{thm}[Tc]{Theorem}
\newtheorem{prop}[Pc]{Proposition}
\newtheorem{lem}[Lc]{Lemma}
\newtheorem{cor}{Corollary}

\title{A Supplement to Scholz's Reciprocity Law}
\author{Franz Lemmermeyer}
\address{M\"orikeweg 1 \\
 73489 Jagstzell \\ Germany}
\email{hb3@ix.urz.uni-heidelberg.de}

\subjclass{Primary 11 R 21; Secondary 11 R 29, 11 R 18}
\begin{document}
\baselineskip=17pt

\begin{abstract}
In this note we will present a supplement to 
Scholz's reciprocity law and discuss applications
to the structure of $2$-class groups of quadratic
number fields.
\end{abstract}

\maketitle

\section{Introduction}

Let us start by fixing some notation:
\begin{itemize}
\item $p$ and $q$ denote primes $\equiv 1 \bmod 4$;
\item $h(d)$ denotes the class number (in the usual sense) 
      of the quadratic number field with discriminant $d$;
\item $\cO_p$ and $\cO_q$ denote the rings of integers in
      $\Q(\sqrt{p}\,)$ and $\Q(\sqrt{q}\,)$;
\item $\eps_p$ and $\eps_q$ denote the fundamental units
      of $\Q(\sqrt{p}\,)$ and $\Q(\sqrt{q}\,)$, respectively;
\item $[\alpha/\fp]$ denotes the quadratic residue symbol in 
      a quadratic number field; recall that it takes values 
      $\pm 1$ and is defined for ideals $\fp \nmid 2\alpha$ 
      by $[\alpha/\fp] \equiv \alpha^{(N\fp-1)/2} \bmod \fp$.
\end{itemize}
Given primes $p \equiv q \equiv 1 \bmod 4$ with $(p/q) = +1$,
we have $p\cO_q = \fp\fp'$ and $q\cO_p = \fq\fq'$; the symbol 
$[\eps_p/\fq]$ does not depend on the choice of $\fq$, so we
can simply denote it by $(\eps_p/q)$. Scholz's reciprocity law 
then says that we always have $(\eps_p/q) = (\eps_q/p)$ (for 
details, see \cite{LRR1, LRR2,LRL}). Scholz's reciprocity law 
was first proved by Sch\"onemann \cite{Schnm}, and then 
rediscovered by Scholz \cite{Sch1} (Scholz mentioned his
reciprocity law and the connection to the parity of the class
number of $\Q(\sqrt{p}, \sqrt{q}\,)$ in a letter to Hasse from 
Aug. 25, 1928; see \cite{LSH}). In \cite{Sch2}, Scholz found 
that in fact $(\eps_p/q) = (\eps_q/p) = (p/q)_4 (q/p)_4$,
and showed that these residue symbols are connected to the 
structure of the $2$-class group of $\Q(\sqrt{pq}\,)$.

\section{Hilbert's Supplementary Laws}
For extending these results we have to recall the notions of
primary and hyper-primary integers (see Hecke \cite{Hecke}).

\begin{lem}\label{Lpr}
Let $K$ be a number field with ring of integers $\cO_K$, and let
$\alpha \in \cO_K$ be an element with odd norm. Then the following
assertions are equivalent:
\begin{enumerate}
\item $\alpha \gg 0$ is totally positive and 
      $\alpha \equiv \xi^2 \bmod 4$ for some $\xi \in \cO_K$;
\item the extension $K(\sqrt{\alpha}\,)/K$ is unramified at all
      primes above $2 \infty$.
\end{enumerate}
\end{lem}

If the conditions of Lemma \ref{Lpr} are satisfied, we say that 
$\alpha$ is primary.

\begin{lem}\label{Lhp}
Assume that $2\cO_K = \fl_1^{e_1} \cdots \fl_r^{e_r}$; then 
the following assertions are equivalent:
\begin{enumerate}
\item $\alpha$ is primary, and 
      $\alpha \equiv \xi^2 \bmod \fl_j^{2e_j+1}$ for all $j$.
\item every prime above $2$ splits in the extension $K(\sqrt{\alpha}\,)/K$.
\end{enumerate}
\end{lem}

If the conditions of Lemma \ref{Lhp} are satisfied, we say that 
$\alpha$ is hyper-primary. Observe that the conditions in (1) are
equivalent to $\alpha \equiv \xi^2 \bmod 4\fl_1 \cdots \fl_r$.
Also note $\alpha$ is allowed to be a 
square in Lemma \ref{Lpr} and Lemma \ref{Lhp}.

Our next result is related to the First Supplementary Law of 
quadratic reciprocity for fields with odd class number; it was
stated and proved in a special case by Hilbert  (\cite{Hil}),
and proved in full generality by Furtw\"angler. Nowadays, this 
result is almost forgotten; for a proof of Hilbert's Supplementary 
Laws (for arbitrary number fields) based on class field theory, see 
\cite{LSRL}; Hecke \cite[Thm 171]{Hecke} gives a proof based on his 
theory of Gauss sums and theta functions over algebraic number fields.

\begin{thm}[Hilbert's First Supplementary Law]
Let $\fa$ be an ideal of odd norm in some number field $k$ with 
odd class number $h$, and let $(\cdot/\cdot)$ denote the quadratic
residue symbol in $\cO_k$. Then the following assertions are equivalent:
\begin{enumerate}
\item $(\eps/\fa) = +1$ for all units $\eps \in \cO_k^\times$;
\item $\fa^h = (\alpha)$ for some primary $\alpha \in \cO_k$.
\end{enumerate}
\end{thm}

Hilbert calls an ideal $\fa$ with odd norm {\em primary} if 
condition (1) above is satisfied, i.e., if $(\eps/\fa) = +1$
for all units $\eps$ in $k$. Hilbert's Second Supplementary Law 
can be given the following form:

\begin{thm}[Hilbert's Second Supplementary Law]\label{T2nd}
Let $\fa$ be a primary ideal of odd norm in some number field $k$ 
with odd class number $h$. Then the following assertions are equivalent:
\begin{enumerate}
\item $(\lambda/\fa) = +1$ for all $\lambda \in \cO_k$
      whose prime divisors consist only of primes above $2$;
\item $\fa^h = (\alpha)$ for some primary $\alpha \in \cO_k$.
\end{enumerate}
\end{thm}

Hilbert calls ideals satisfying property (1) above {\em hyper-primary}.
A proof of a generalization of Thm. \ref{T2nd} to arbitrary number
fields can be found in \cite[Thm 175]{Hecke}. Now we can state

\begin{thm}\label{P1}
Let $p \equiv q \equiv 1 \bmod 4$ be primes with $(p/q) = +1$. 
Then $p\cO_q = \fp\fp'$ and $q\cO_p = \fq\fq'$ split. 
The class numbers $h(p)$ and $h(q)$ are odd, and there exist
elements $\pi \in \cO_q$ and $\rho \in \cO_p$ such that 
$\fp^{h(p)} = (\pi)$ and $\fq^{h(q)} = (\rho)$.

Then the following assertions are equivalent:
\begin{enumerate}
\item $(\eps_p/q) = +1$;
\item $\rho$ can be chosen primary;
\item $h(pq) \equiv 0 \bmod 4$.
\end{enumerate}
\end{thm}

\begin{proof}
Genus theory (see e.g. \cite[Chap. 2]{LRL}) implies that 
$h(p) \equiv h(q) \equiv 1 \bmod 2$. 
The equivalence (1) $\Llra$ (2) is a special case of Hilbert's 
First Supplementary Law for fields with odd class number 
(\cite{Hil}); observe, however, that Hilbert stated and proved
this law only for a very narrow class of fields -- the general
statement was proved only by Furtw\"angler. The equivalence 
(1) $\Llra$ (3) is due to Scholz \cite{Sch2}. 

It is not hard to prove these statements directly using class
field theory; below we will do this in an analogous situation.
\end{proof}

Observe that part (3) of Thm. \ref{P1} is symmetric in $p$ and 
$q$, which immediately implies Scholz's reciprocity law 
$(\eps_p/q) = (\eps_q/p)$. Note that we can state this
reciprocity law in the following form:

\begin{cor}
Let $p$ and $q$ satisfy the assumptions of Thm. \ref{P1}. If the 
ideals above $q$ in $\Q(\sqrt{p}\,)$ are primary, then so are
the ideals above $p$ in $\Q(\sqrt{q}\,)$.
\end{cor}

In the next section we will prove an analogous result connected 
to Hilbert's Second Supplementary Law of Quadratic Reciprocity.

\section{A Supplement to Scholz's Reciprocity Law}.

Assume that $p \equiv q \equiv 1 \bmod 8$ are primes. Then 
$2$ splis in $\Q(\sqrt{p}\,)$ and $\Q(\sqrt{q}\,)$, and we
can write $2\cO_p = \fl\fl'$ and $2\cO_q = \fm\fm'$. Now
pick elements $\lambda_p, \lambda_q$ such that 
$\fl^{h(p)} = (\lambda_p)$ and $\fm^{h(q)} = (\lambda_q)$. 
Since both fields have units with independent signatures,
we may assume that $\lambda_p, \lambda_q \gg 0$. The
quadratic residue symbol $[\lambda_p/\fq]$, where
$q\cO_p = \fq\fq'$, does not depend on the choice of
$\lambda_p$ or $\fq$, so we may denote it by $(\lambda_p/q)$.

\begin{thm}\label{T1}
Let $p \equiv q \equiv 1 \bmod 8$ be primes with $(p/q) = +1$, 
and assume that $(\eps_p/q) = (\eps_q/p) = +1$. Then the following
assertions are equivalent:
\begin{enumerate}
\item $(\lambda_p/q) = +1$;
\item $\rho$ can be chosen hyper-primary;
\item the ideal classes generated by the ideals above $2$
      in $F = \Q(\sqrt{pq}\,)$ are fourth powers in $\Cl(F)$;
\end{enumerate}
\end{thm}

\begin{proof}
Let $F = \Q(\sqrt{pq}\,)$; then $F_1 = F(\sqrt{p}\,)$ is an
unramified quadratic extensions; since $\rho$ is primary, the 
extension $F(\sqrt{\rho}\,)/F$ is unramified, and it is easily
checked that it is the unique cyclic quartic unramified extension 
of $F$. Since $2$ splits completely in $\Q(\sqrt{p},\sqrt{q}\,)/\Q$,
it will split completely in $F(\sqrt{\rho}\,)/\Q$ if and only
if $2$ splits completely in $\Q(\sqrt{p}, \sqrt{\rho}\,)$, which
happens if and only if $\rho$ is hyperprimary. On the other hand,
the decomposition law in unramified abelian extensions shows that
the prime ideals above $2$ split completely in $F(\sqrt{\rho}\,)/F$ 
if and only if their ideal classes are fourth powers in $\Cl(F)$.
This proves that (2) $\Llra$ (3).

The equivalence (1) $\Llra$ (2) is a special case of the Second 
Supplementary Law of Hilbert's Quadratic Reciprocity Law in 
number fields with odd class number. Here is a direct argument
using class field theory.

Consider the quadratic extension $K = F(\sqrt{\rho}\,)$ of $F$.
Then $\rho$ is hyper-primary if and only if the prime $\fl$
(and, therefore, also its conjugate $\fl'$) above $2$ splits 
in $K/F$. Since $K$ is the unique quadratic subextension of
the ray class field modulo $\fq$ over $F$, which has degree $2h(p)$, 
the prime $\fl$ will split in $K/F$ if and only if
$\fl^{h(p)} = (\lambda_p)$ for some $\lambda_p \equiv \xi^2 \bmod \fq$.
This shows that (1) $\Llra$ (2).
\end{proof}

The symmetry of $p$ and $q$ in the third statement of Thm. \ref{T1}
then implies

\begin{cor}\label{C1}
We have $(\lambda_p/q) = (\lambda_q/p)$.
\end{cor}

While the proof of Thm. \ref{T1} required class field theory, the 
actual reciprocity law in Cor. \ref{C1} can be proved with elementary 
means. We will now give a proof \`a la Brandler \cite{Bra}. To this end, 
write $\lambda_p = \frac{a+b\sqrt{p}}2$; then $a^2 - pb^2 = 2^u$,
where $u = h(p)+2 = 2m+1$ is odd. From $a^2-2^u = pb^2$ we find 
that $a + 2^m\sqrt{2} = \pi_2 \beta^2$ and 
$a - 2^m\sqrt{2} = \pi_2' {\beta'}^2$, where $\pi_2 \pi_2' = p$
for some totally positive $\pi_2 \equiv 1 \bmod 2$. Moreover 
$\beta\beta' = b$ and  $2a = \pi\beta^2 + \pi' {\beta'}^2$. 

Now $(\pi_2 \beta + \beta'\sqrt{p}\,)^2 
     = \pi_2(\pi\beta^2 + \pi_2' {\beta'}^2 + 2y\sqrt{p}\,)
     = 2\pi_2\lambda$. Standard arguments then show that 
$[\frac{\pi_2}{\rho_2}] = (\frac{\lambda}{q})$, where $\rho_2\rho_2' = q$.

The quadratic reciprocity law in $\Z[\sqrt{2}\,]$ shows 
that $[\frac{\pi_2}{\rho_2}] = [\frac{\rho_2}{\pi_2}]$, and his implies
the following elementary form of the supplement to Scholz's 
reciprocity law:
$$ \Big(\frac{\lambda_p}{q}\Big) = 
       \Big[\frac{\pi_2}{\rho_2}\Big] = \Big[\frac{\rho_2}{\pi_2}\Big] =
   \Big(\frac{\lambda_q}{p}\Big). $$

\section{Additional Remarks}
We close this article with a few remarks and questions.

\medskip
\noindent{\bf Remark 1.}
Since $p \equiv q \equiv 1 \bmod 8$, we can also write
$p = N\pi_2^*$ and $q = N\rho_2^*$ for elements 
$\pi_2^*, \rho_2^* \in \Z[\sqrt{-2}\,]$ with 
$\pi_2^* \equiv \rho_2^* \equiv 1 \bmod 2$. Then \cite[Prop. 2]{Lnc}
states that 
$$ \Big[\frac{\pi_2}{\rho_2}\Big]\Big[\frac{\pi_2^*}{\rho_2^*}\Big]
        = \Big(\frac pq\Big)_4^{\phantom{p}} 
          \Big(\frac qp\Big)_4^{\phantom{p}}. $$
Under the assumptions of Thm. \ref{T1}, this means that
$$ \Big(\frac{\lambda_p}{q}\Big) = 
   \Big(\frac{\lambda_q}{p}\Big) = 
   \Big[\frac{\pi_2}{\rho_2}\Big] = \Big[\frac{\pi_2^*}{\rho_2^*}\Big]. $$

\medskip\noindent
{\bf Remark 2.}
Hilbert's Supplementary law as we have stated it applies to all
(quadratic) fields with odd class number, not just the fields
with prime discriminant. Here we give an example that shows what
to expect in this more general situation.

Consider primes $p \equiv q \equiv 3 \bmod 4$ and primes 
$r \equiv 1 \bmod 4$ with $(pq/r) = +1$. Let $\eps_{pq}$ denote 
the fundamental unit in $k = \Q(\sqrt{pq}\,)$. Then the prime
ideals $\fr$ and $\fr'$ above $r$ in $k$ satisfy 
$\fr^{h(pq)} = (\rho)$ for some primary $\rho$
if and only if $(\eps_{pq}/r) = +1$. Since $p\eps_{pq}$ 
is a square in  $k$, we have $(\eps_{pq}/r) = (p/r)$.

Assume now that $\rho$ can be chosen primary, and consider the dihedral 
extension $L/\Q$ with $L = \Q(\sqrt{p}, \sqrt{q}, \sqrt{\rho}\,)$. 
Clearly $\rho$ is primary if and only if $L/\Q(\sqrt{pqr}\,)$ is cyclic 
and unramified. It is then easy to show that the quadratic extensions
of $\Q(\sqrt{r}\,)$ different from $\Q(\sqrt{pq},\sqrt{r}\,)$ can
be generated by a primary element $\alpha$ with prime ideal
factorization $(\fp\fq)^{h(r)}$ for a suitable choice of prime
ideals $\fp$ and $\fq$ above $p$ and $q$, respectively. Note that
if $\fp\fq$ is primary, then $\fp\fq'$ is not, since $\fq\fq'=(q)$
is not primary (we have either $q < 0$ or $q \equiv 3 \bmod 4$).

The upshot of this discussion is: if $\rho$ is primary, then 
exactly one of the ideals $\fp\fq$ and $\fp\fq'$ is primary,
say the first one, and then Hilbert's Supplementary Law 
shows that $(\eps_r/pq) := [\eps_r/\fp\fq] = +1$. Conversely,
if $\fp\fq$ is primary, then $(\eps_r/pq) = (\eps_{pq}/r) = +1$.
We have shown:

\begin{prop}
Let $p \equiv q \equiv 3 \bmod 4$ and $r \equiv 1 \bmod 4$ be primes
with $(pq/r) = +1$. Then the following assertions are equivalent:
\begin{enumerate}
\item $(\eps_{pq}/r) = +1$;
\item $(p/r) = +1$;
\item the ideal $\fr$ in $\Q(\sqrt{pq}\,)$ above $r$ is primary;
\item $h(pqr) \equiv 0 \bmod 4$;
\item there is a unique primary ideal $\fa$ (up to conjugation) of norm $pq$
      in $\Q(\sqrt{r}\,)$, and $(\eps_r/pq) := [\eps_r/\fa] = +1$.
\end{enumerate}
\end{prop}

Note that $(\eps_r/pq)$ is not well defined if $(p/r) = -1$ since
in this case we do not have a canonical way to single out 
the prime ideals above $p$ and $q$ in  $\Q(\sqrt{r}\,)$.

As an example, consider the case $p=3$, $q=7$, $r = 37$; then the 
elements of norm $21$ in the ring of integers in $\Q(\sqrt{37}\,)$
are $\pm 13 \pm 2\sqrt{37}$ (these elements are not primary:
the element $-13 + 2\sqrt{37} \equiv 1 \bmod 4$ is not totally 
positive) and $\frac{\pm 11 \pm \sqrt{37}}2$. It is easy to 
check that $\beta = \frac{11 + \sqrt{37}}2$ is primary; now
$\eps_r = 6 + \sqrt{37}$, and 
$[\eps_r/\beta] = (\frac{-5}{21}) = +1$ as claimed, whereas
$[\eps_r/(13 \pm 2 \sqrt{37}\,] = -1$.

\medskip\noindent
{\bf Remark 3.} Above we have seen that, under suitable assumptions,
the ideal class generated by a prime above $2$ in $\Q(\sqrt{pq}\,)$ 
is a fourth power in the class group if and only if 
$[\frac{\pi_2}{\rho_2}] = +1$, where $\pi_2, \rho_2 \in \Z[\sqrt{2}\,]$
are elements $\equiv 1 \bmod 2$ with norms $p$ and $q$, respectively.
Does an analogous statement hold with $2$ replaced by an odd
prime $\ell \ne p, q$?

\medskip\noindent
{\bf Remark 4.} 
Budden, Eisenmenger \& Kish \cite{BEK} have generalized 
Scholz's reciprocity law to higher powers; can the reciprocity law 
$(\lambda_p/q) = (\lambda_q/p)$ proved above also be generalized
in this direction?


\begin{thebibliography}{999}

\bibitem{Bra} J. Brandler,
{\em Residuacity properties of real quadratic units},  
J. Number Theory  {\bf  5} (1973), 271--287
%

\bibitem{BEK} M. Budden, J. Eisenmenger, J. Kish,
{\em A Generalization of Scholz's Reciprocity Law},
J. Th\'eorie Nombr. Bordeaux, to appear
%

\bibitem{Hecke} E. Hecke,
{\em Lectures on the theory of algebraic numbers},
Springer-Verlag 1981
%

\bibitem{Hil} D. Hilbert,
{\em \"Uber die Theorie des relativquadratischen Zahlk\"orpers},
Math. Ann. {\bf 51} (1899), 1--127.
%

\bibitem{LRR1} F.~Lemmermeyer,
{\em Rational quartic reciprocity},  
Acta Arith.  {\bf 67} (1994),  no. 4, 387--390
%

\bibitem{LRR2} F.~Lemmermeyer,
{\em Rational quartic reciprocity. II},  
Acta Arith.  {\bf 80} (1997),  no. 3, 273--276
%

\bibitem{LRL} F. Lemmermeyer,
{\em Reciprocity laws. From Euler to Eisenstein},
Springer-Verlag, 2000
%

\bibitem{Lnc} F. Lemmermeyer,
{\em Some families of non-congruent numbers},
Acta Arith. {\bf 110} (2003), 15--36
%

\bibitem{LSRL} F. Lemmermeyer,
{\em Selmer Groups and Quadratic Reciprocity},
Abh. Math. Sem. Hamburg {\bf 76} (2006), 279--293
%

\bibitem{LSH} F. Lemmermeyer,
{\em Die Korrespondenz Hasse--Scholz},
in preparation
%
 
\bibitem{Schnm} Th. Sch\"onemann,
{\em Ueber die Congruenz $x^2 + y^2 \equiv 1 \pmod p$},
J. Reine Angew. Math. {\bf 19} (1839), 93--112
%

\bibitem{Sch1} A. Scholz,
{\em Zwei Bemerkungen zum Klassenk\"orperturm},
J. Reine Angew. Math. {\bf 161} (1929), 201--207
%

\bibitem{Sch2} A. Scholz,
{\em \"Uber die L\"osbarkeit der Gleichung $t^2 - Du^2 = -4$},
Math. Z. {\bf 39} (1934), 95--111
%

\end{thebibliography}
\end{document}